\documentclass[11pt]{article}
\usepackage{amsmath, epsfig}
\usepackage{latexsym}
\usepackage{mathrsfs}
\usepackage{amssymb}
\usepackage{graphicx}
\usepackage{color}



\newtheorem{thm}{Theorem}[section]
\newcommand{\n}{\noindent}

\newtheorem{conj}[thm]{Conjecture}
\newtheorem{lemma}[thm]{Lemma}

\newtheorem{cor}[thm]{Corollary}

\newenvironment{proof}{{\bf Proof}.}{\rule{3mm}{3mm}}
\begin{document}

\title{$Z_3$-connectivity with independent number 2 }
\author{Fan Yang \thanks{School of Mathematics and Physics, Jiangsu University of Science and Technology,
Zhenjiang, Jiangsu 212003, China. E-mail: fanyang$\_$just@163.com. Research is partially supported by NSF-China Grant:
NSFC 11326215. and is partially supported by NSF-China Grant:
NSFC 11371009.},
Xiangwen Li\thanks{Department of Mathematics
 Huazhong Normal
University, Wuhan 430079, China, Research is partially supported by NSF-China Grant:
NSFC: 11171129.},
Liangchen Li\thanks{School of Mathematical Sciences, Luoyang Normal University,
Luoyang 471022, China, Research is partially supported by NSF-China Grant:
NSFC: 11301254.}
}

\date{}
\maketitle

\begin{abstract}
Let $G$ be a 3-edge-connected graph on $n$ vertices. It is
proved in this paper that if $\alpha (G)\le 2$, then either $G$ can be
$Z_3$-contracted to one of graphs $\{K_1, K_4\}$ or $G$ is one of the graphs in Fig. 1.
\end{abstract}

\section{Introduction}

Graphs considered here are undirected, finite and may have multiple edges without
loops\cite{Bondy}. Let $G$ be a graph.
Set $D=D(G)$ be an orientation of $G$. If an edge $e=uv\in E(G)$ is
directed from a vertex $u$ to a vertex $v$, then $u$ is a tail of
$e$, $v$ is a head of $e$. For a vertex $v\in V(G)$, let
$E^+(v)(E^-(v))$ denote the set of all edges with $v$ as a tail(a
head). Let $A$ be an abelian group with the additive identity 0, and
let $A^*=A-\{0\}$.

For every mapping  $f: E(G)\rightarrow A$, the boundary of $f$ is a
function $\partial f: V(G)\rightarrow A$ defined by
\[
\partial f(v)=\sum\limits_{e\in E^{+}(v)} f(e)- \sum\limits_{e\in
E^-(v)} f(e),
\]
where ``$\sum$" refers to the addition in $A$. If $\partial f(v)=0$
for each vertex $v\in V(G)$, then $f$ is called an $A$-flow of $G$.
Moreover, if $f(e)\neq 0$ for every $e\in E(G)$, then $f$ is a
nowhere-zero $A$-flow of $G$.

A graph $G$ is {\it $A$-connected} if for any mapping $b: V(G)\rightarrow
A$ with $\sum\limits_{v\in V(G)} b(v)=0$, there exists an
orientation of $G$ and a mapping  $f: E(G)\rightarrow A^*$ such that
$\partial f(v)=b(v)$ (mod 3) for each $v\in V(G)$. The concept of
$A$-connectivity was firstly introduced by Jaeger et al in
\cite{F.Jaeger} as a generalization of nowhere-zero flows.
Obviously, if $G$ is $A$-connected, then $G$ admits a nowhere-zero
$A$-flow.

For $X\subseteq E(G)$, the contraction $G/X$ is obtained from $G$ by
contracting each edge of $X$ and deleting the resulting loops. If
$H\subseteq G$, we write $G/H$ for $G/E(H)$. Let $A$ be an abelian group with $|A|\ge 3$.
Denote by $G'$ the  graph obtained by repeatedly contracting $A$-connected
subgraphs of $G$ until no such subgraph left.  We  say $G$ can be
$A$-contracted to $G'$.  Clearly, if a graph $G$ can be $A$-contracted
to $K_1$, then  $G$ is $A$-connected.

In this paper, we focus on $Z_3$-connectivity. The following conjecture is due to
 Jaeger {\it et al.}

\begin{conj}\cite{F.Jaeger}
Every 5-edge-connected graph is $Z_3$-connected.
\end{conj}

It is still open. However, many authors are devoted to approach  this conjecture.
Chv\'{a}tal and Erd\H{o}s \cite{Chvatal} proved a classical result: a graph $G$ with at least 3 vertices is hamiltonian
if its independence number is less than or equal to its connectivity (this condition is known as
Chv\'{a}tal-Erd\H{o}s Condition). Therefore Chv\'{a}tal-Erd\H{o}s Condition guarantees the existence of
nowhere-zero 4-flows. Recently, Luo, Miao, Xu \cite{Luo M} characterized the graphs satisfying Chv\'{a}tal-Erd\H{o}s Condition that admit a nowhere-zero 3-flow.


\begin{center}
\includegraphics[bb=0 0 487 376, height=7 cm]{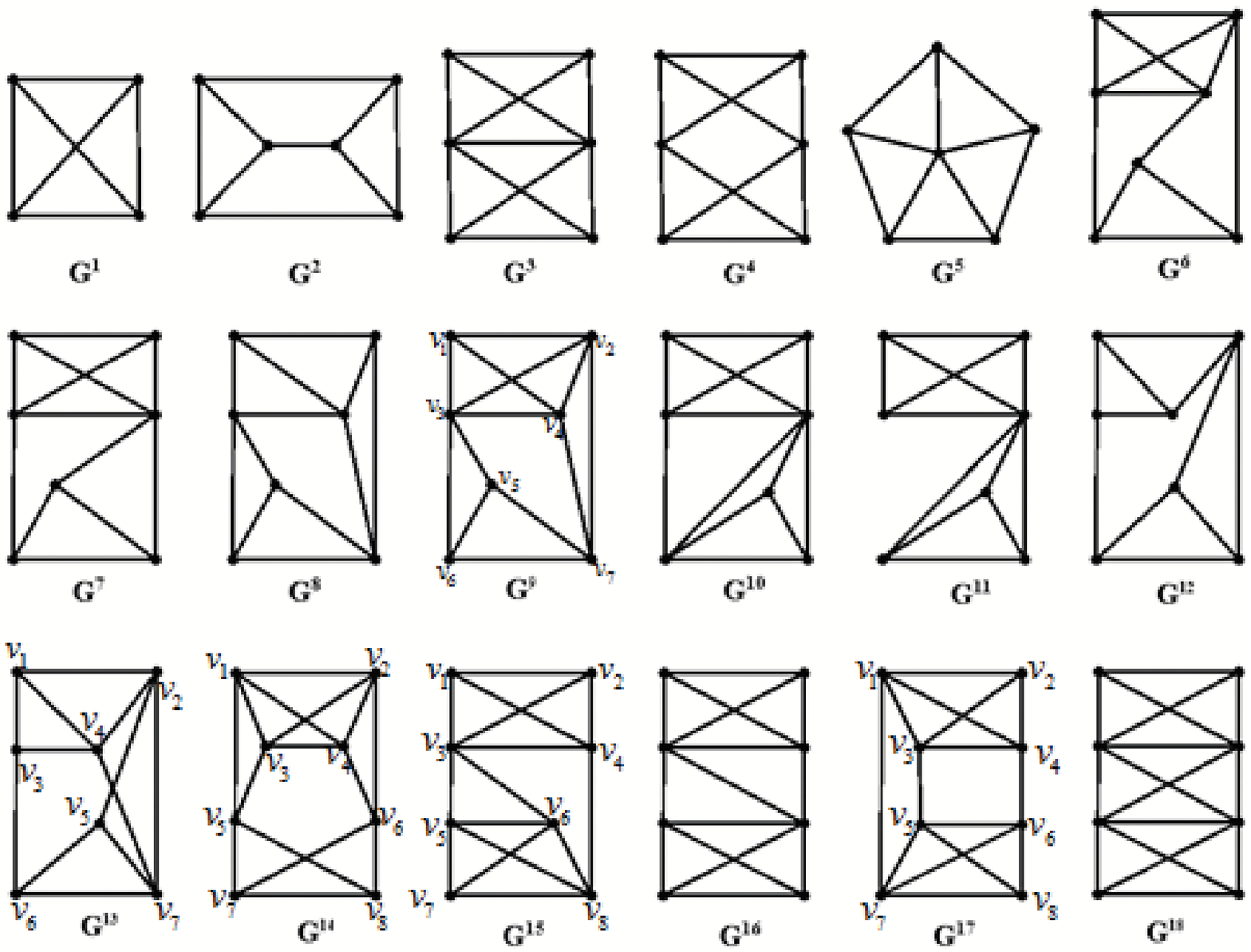}
\vskip 0.3cm
\begin{picture}(15,1)
\put(0.5,1){Fig. 1:  18 specified graphs which is $Z_3$-connected}
\end{picture}
\end{center}

\begin{thm}\label{luo}{(Luo et al. \cite{Luo M})}
Let $G$ be a bridgeless graph with independence number $\alpha(G)\le 2$. Then $G$ admits a nowhere-zero 3-flow if and
only if $G$ can not be contracted to a $K_4$ and $G$ is not one of $G^3, G^5, G^{18}$ in Fig. 1 or $G\notin {G^3}'$.
\end{thm}

Motivated by this, we consider the $Z_3$-connectivity of graphs satisfying the weaker Chv\'{a}tal-Erd\H{o}s Condition.
In this paper, we extend Luo {\it et al.}'s result to group connectivity. The
main theorem  is as follows.

\begin{thm}
\label{th1} Let $G$ be a 3-edge-connected simple graph and $\alpha(G)\le 2$.
$G$ is not one of the 18 special graphs shown in Fig. 1
if and only if  $G$ can be $Z_3$-contracted to one of the graphs $\{K_1,
K_4\}$.
\end{thm}

From Theorem~\ref{th1}, we obtain the following corollary immediately.

\begin{cor}
Let $G$ be a 3-edge-connected graph and $\alpha(G)\le 2$.
Then one of the following holds:

(i) $G$ can be $Z_3$-contracted to one of the graphs $\{K_1,
K_4\}$, or

(ii) $G$ is one of the 18 special graphs shown in Fig. 1, or

(ii) $G$ is one of the graphs $\{{G^3}', {G^4}', {G^{10}}', {G^{11}}'\}$ shown in Fig. 2,
where $u, v$ are adjacent by $m$ edges, $m\ge 2$ for $i=3, 4, 10$ and $m\ge 3$ for $i=11$.
\end{cor}

\includegraphics[bb=-200 0 725 474, height=6 cm]{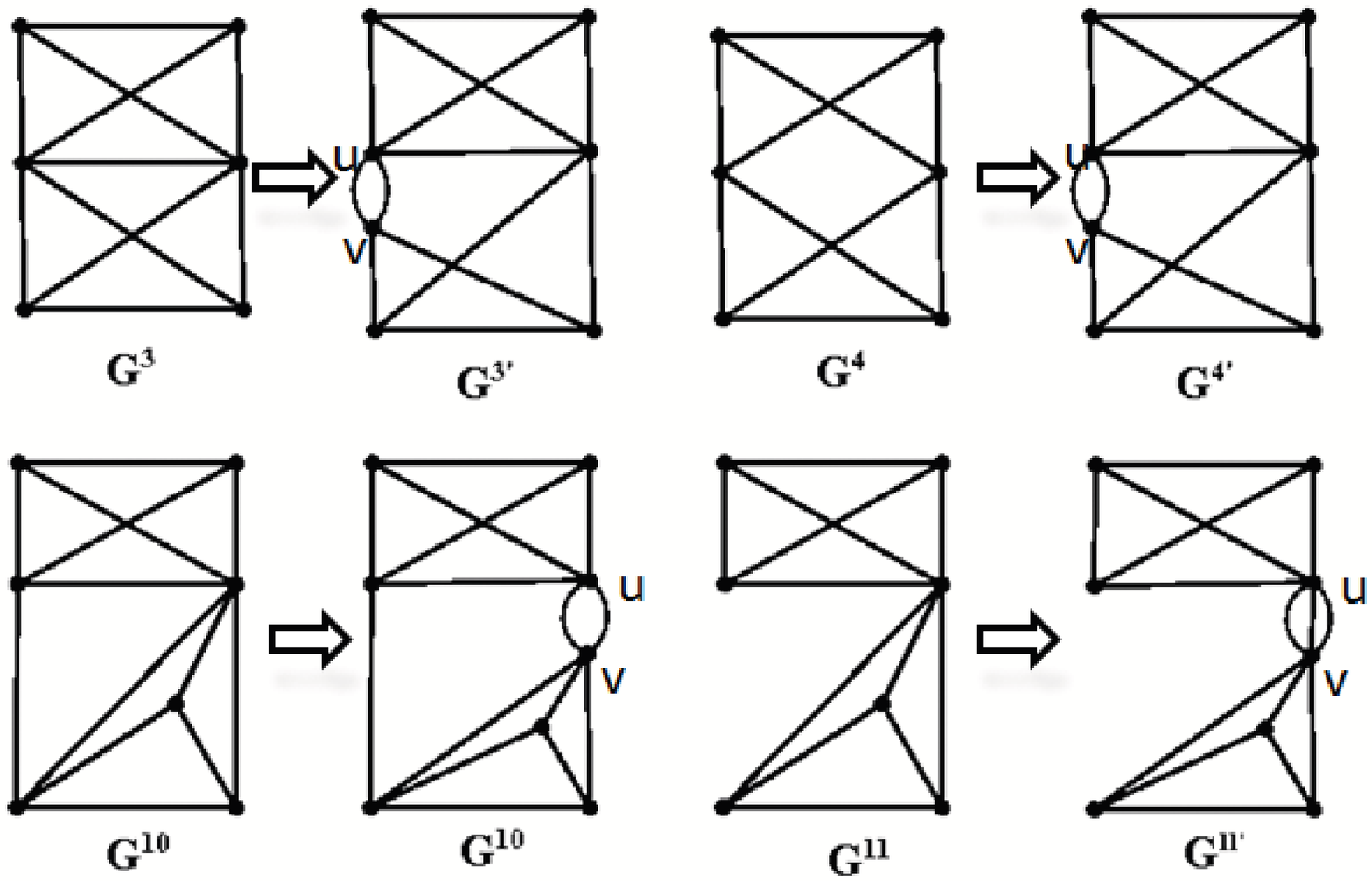}
\vskip 0.3cm
\begin{picture}(15,1)
\put(5,1){Fig. 2: Construction of graph of ${G^3}'$, ${G^4}'$, ${G^{10}}'$, ${G^{11}}'$}
\end{picture}

We end this section with some terminology and notation not define in \cite{Bondy}. For $V_1, V_2\subseteq V(G)$ and $V_1\cap V_2=\emptyset$,
denote by $e(V_1, V_2)$ the number of edges with one endpoint in
$V_1$ and the other endpoint in $V_2$. For $S\subseteq V(G)$, $G[S]$ denotes an induced subgraph of $G$ with vertex-set $S$.
Let $N_G(v)$ denote the set of
all vertices adjacent to vertex $v$; set $N_G[v]=N_G(v)\cup \{v\}$.
We usually use $N(v)$ and $N[v]$ for $N_G(v)$ and $N_G[v]$ if there
is no confusion.  A $k$-vertex denotes a vertex of
degree $k$. Let $K_n$ denote a complete graph with $n$
vertices, where $n\geq 3$. Moreover, $K_3$ denotes a 3-cycle. A {\it
$k$-cycle} is a cycle of length $k$; a 3-cycle is also called a {\it
triangle.} The wheel $W_k$ is the graph obtained from a $k$-cycle by
adding a new vertex and joining it to every vertex of the $k$-cycle.
When $k$ is odd (even), we say $W_k$ is an {\it odd (even)} wheel. For
convenience, we define $W_1$ as a triangle.

\section{Preliminary}

Here we state some lemmas which are essential to the proof of our result.

\begin{lemma}
\label{o} Let $A$ be an abelian group with $|A|\ge 3$. The following
results are known:

(1) (Proposition 3.2 of \cite{H.J.Lai}) $K_1$ is $A$-connected;

(2)(Corollary 3.5 of \cite{H.J.Lai}) $K_n$ and $K_n^-$ are $A$-connected if $n\ge 5$;

(3) (\cite{F.Jaeger} and Lemma 3.3 of \cite{H.J.Lai})$C_n$ is $A$-connected if and only if $|A|\ge n+1$;

(4) (Theorem 4.6 of \cite{J.J.Chen}) $K_{m,n}$ is $A$-connected if $m\ge n\ge 4$; neither $K_{2,t}$
$(t\ge 2)$ nor $K_{3,s}$ $(s\ge 3)$ is $Z_3$-connected;

(5) (Lemma 2.8 of \cite{J.J.Chen} and Proposition 2.4 of \cite{M.DeVos} and Lemma 2.6 of \cite{G.Fan}) Each even wheel is $Z_3$-connected and each odd wheel is not;

(6) (Proposition 3.2 of \cite{H.J.Lai}) Let $H\subseteq G$ and $H$ be $A$-connected. $G$ is
$A$-connected if and only if $G/H$ is $A$-connected;

(7) (Lemma 2.3 of \cite{hou}) Let $v$ be not a vertex of
$G$. If $G$ is $A$-connected and $e(v, G)\ge 2$, then $G\cup \{v\}$
is $A$-connected.
\end{lemma}

Let $G$ be a graph and $u,v, w$ be three vertices of $G$ with $uv,
uw\in E(G)$, and $d_G(u)\ge 4$. Let $G_{[uv,uw]}$ be the graph
$G\cup \{vw\}-\{uv, uw\}$.

\begin{lemma}
\label{1}{(Theorem 3.1 of \cite{J.J.Chen})} Let $A$ be an abelian group with $|A|\geq 3$. If
$G_{[uv,uw]}$ is $A$-connected, then so is $G$.
\end{lemma}
Let $H_1$ and $H_2$ be two disjoint graphs. The 2-sum of $H_1$ and $H_2$,
denoted by $H_1\oplus H_2$, is the grpah obtained from $H_1\cup H_2$ by identifying exactly one edge.
A graph $G$ is triangularly connected if for any two distinct edges $e$, $e'$, there is a sequence of distinct
cycles of length at most 3, $C_1, C_2, \ldots, C_m$ in $G$ such that $e\in E(C_1)$, $e'\in E(C_m)$ and
$|E(C_i)\cap E(C_{i+1})|=1$ for $1\le i\le m-1$.

\begin{lemma}\label{fan}(Fan et al.\cite{G.Fan})
Let $G$ be a triangularly connected graph. Then $G$ is $A$-connected for all abelian group
$A$ with $|A|\ge 3$ if and only if $G\neq H_1\oplus H_2\oplus \ldots \oplus H_k$,
where $H_i$ is an odd wheel (including a triangle) for $1\le i\le k$.
\end{lemma}

\includegraphics[bb=-50 0 810 330, height=6 cm]{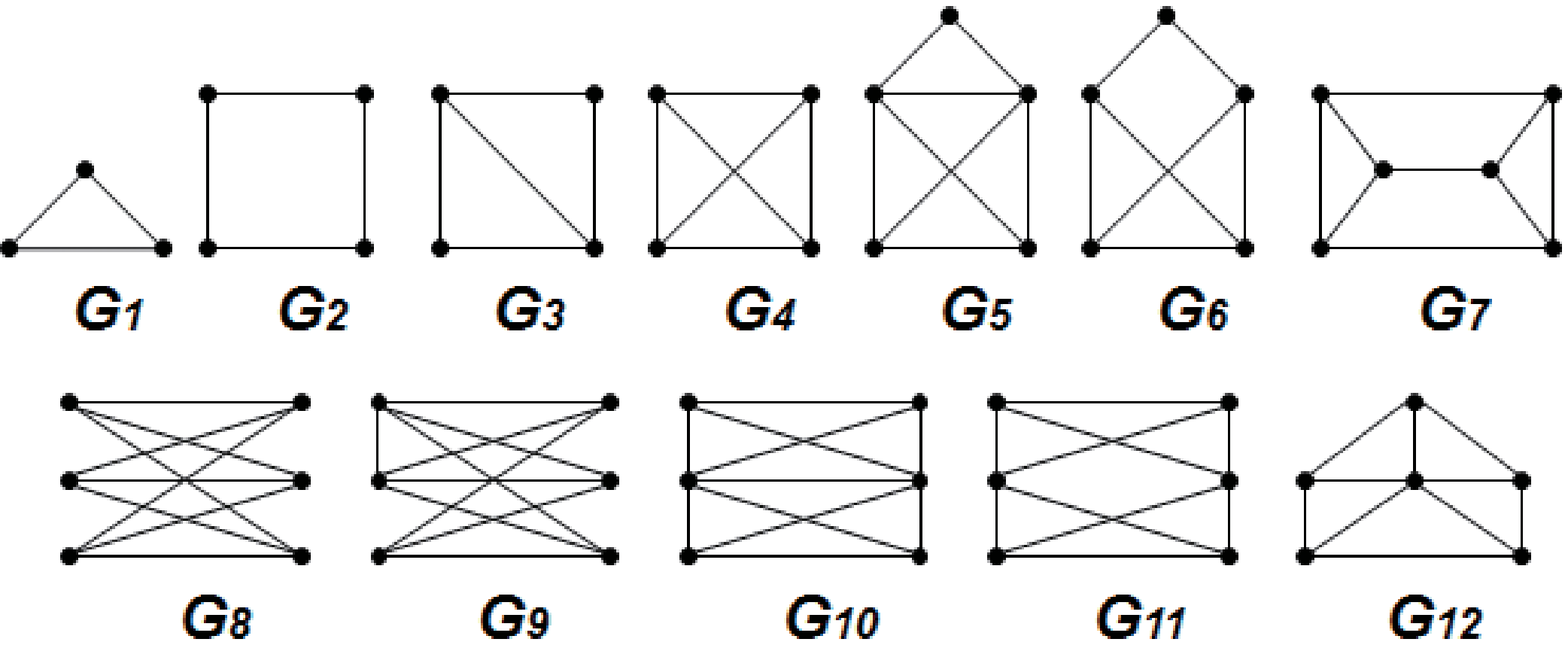}
\vskip 0.3cm
\begin{picture}(15,1)
\put(10,1){Fig. 3:  12 specified graphs}
\end{picture}

There are lots of results about Degree condition and $Z_3$-connectivity. We say
$G$ satisfies Ore-condition, if for each $uv\notin E(G)$, $d(u)+d(v)\ge |V(G)|$. We will
discuss our result via the following Theorem.

\begin{thm}
\label{ore}(Luo et al. \cite{Luo}) A simple graph $G$ satisfying the
Ore-condition
with at least 3 vertices is not $Z_3$-connected if and only if $G$ is one of the
12 graphs in  Fig. 3.
\end{thm}

\begin{lemma}
\label{not Z}
Let $G$ be a graph. If for some mapping $b: V(G) \rightarrow Z_3$ with $\sum_{v\in V(G)} b(v)=0$, there exists no orientation
such that $|E^+(v)|-|E^-(v)|=b(v)$ (mod 3) for each $v\in V(G)$, then $G$ is not $Z_3$-connected.
\end{lemma}
\begin{proof}
By the definition of $Z_3$-connectivity, we know that $G$ is $Z_3$-connected if and only
if for any $b: V(G) \rightarrow Z_3$ with $\sum_{v\in V(G)} b(v)=0$, there exists an orientation and function $f: E(G)\rightarrow Z_3^*$
such that  $\partial f(v)=b(v)$ (mod 3) for each $v\in V(G)$. We know that $Z_3$-connectivity is independent on the orientation of graph.
For above $b$ and $f$, We only need to focus on edges of $f(e)=2$.
If $f(e)=2$, then we can invert
the orientation of $e$ and let $f(e)=1$, the others maintain. In this way we can get a new orientation of $G$ and a new function $f'$ on $E(G)$,
such that $\partial f'(v)=b(v)$ for each $v\in V(G)$.
Follow this way, finally we can get a new orientation and a new function $f'': E(G)\rightarrow Z_3$
such that $f''(e)=1$ and $\partial f''(v)=b(v)$ for each $e\in E(G)$ and each $v\in V(G)$.
Thus we can deduce that $G$ is $Z_3$-connected if and only
if for any $b: V(G) \rightarrow Z_3$ with $\sum_{v\in V(G)} b(v)=0$, there exists an orientation and function $f: E(G)\rightarrow Z_3^*$
such that  $f(e)=1$ and $\partial f(v)=b(v)$ for each $e\in E(G)$ and each $v\in V(G)$.
In this case, $\partial f(v)=|E^+(v)|-|E^-(v)|$ since $f(e)=1$. That is,
$G$ is $Z_3$-connected if and only if for any $b: V(G) \rightarrow Z_3$ with $\sum_{v\in V(G)} b(v)=0$, there exists an orientation
such that $|E^+(v)|-|E^-(v)|=b(v)$ for each $v\in V(G)$.
We are done.
\end{proof}

\begin{lemma}
\label{g6}
If $G$ is one of graphs in Fig. 1, then $G$ is not $Z_3$-connected.
\end{lemma}
\begin{proof}
By Theorem~\ref{ore},  each graph in $\{G^1, G^2, G^3, G^4, G^5\}$
is not $Z_3$-connected.

If $G\cong G^{9}$, then let $b(v)=0$ for each 3-vertex and 5-vertex in $G$ and $b(v)=1$ for each 4-vertex in $G$.
By Lemma~\ref{not Z}, we only need to proof that there exists no orientation such that $|E^+(v)|-|E^-(v)|=b(v)$ (mod 3) for each $v\in V(G)$.
Since $b(v_5)=b(v_6)=0$, we can orient edges such that $E^+(v_5)=0$ and $E^+(v_6)=3$ or $E^+(v_5)=3$ and $E^+(v_6)=0$.
In the former case, we can orient edges $v_3v_1, v_3v_2, v_3v_4$ all with $v_3$ as a tail (or all with $v_3$ as a head) by $b(v_3)=0$.
WLOG, we assume edges $v_3v_1, v_3v_2, v_3v_4$ all with $v_3$ as a tail. Then $E^+(v_1)=0$ by $b(v_1)=0$.
In this case, we must orient edges $v_2v_4$, $v_4v_7$ both with $v_4$ as a head. But we cannot orient $v_2v_7$, such that $|E^+(v_7)|-|E^-(v_7)|=1$ $($mod 3$)$.
The proof of the latter case is similar as above.
Thus $G^{9}$ is not $Z_3$-connected by Lemma~\ref{not Z}.
Since $G^8$ is a spanning subgraph of $G^9$, $G^8$ is not $Z_3$-connected by Lemma~\ref{o} (1) (6).

By Lemma~\ref{fan}, each graph of $\{G^{10}$, $G^{16}, G^{18}\}$ is not $Z_3$-connected.

Since $G\in \{G^7, G^{11}, G^{12}\}$ is a spanning subgraph of $G^{10}$, each graph of $\{G^7$, $G^{11}, G^{12}\}$ is not $Z_3$-connected by Lemma~\ref{o} (1) (6).

If $G\cong G^{13}$, then let $b(v)=0$ for each 3-vertex in $G$ and $b(v)=1$ for each 4-vertex in $G$.
That is $b(v_1)=b(v_3)=b(v_5)=b(v_6)=0$, $b(v_2)=b(v_4)=b(v_7)=1$.
We first orient edges adjacent to vertices with $b(v)=0$. Then in either case, we can orient edges $\{v_4v_2, v_4v_7\}$ both with $v_4$ as a head by $b(v_4)=1$
and orient edges $\{v_4v_7, v_2v_7\}$ both with $v_7$ as a head by $b(v_7)=1$. That is edge $v_4v_7$ has two orientations, a contradiction.
By Lemma~\ref{not Z}, $G^{13}$ is not $Z_3$-connected.

If $G\cong G^{14}$, then let $b(v)=0$ for each 3-vertex in $G$ and $b(v)=1$ for each 4-vertex in $G$.
We first orient edges adjacent to vertices with $b(v)=0$. WLOG, we assume $E^+(v_7)=3$ and $E^+(v_8)=0$.
Then we can orient edges $\{v_1v_5, v_3v_5\}$ both with $v_5$ as a head by $b(v_5)=1$
and orient edges $\{v_2v_6, v_4v_6\}$ both with $v_6$ as a head by $b(v_6)=1$.  Then we can orient edges
$\{v_1v_3, v_2v_3, v_4v_3\}$ all with $v_3$ as a head (a tail) by $b(v_3)=1$ and orient edges
$\{v_1v_4, v_2v_4, v_4v_3\}$ all with $v_4$ as a tail (a head) by $b(v_4)=1$.
In either case, we cannot orient edge $v_1v_2$, such that $|E^+(v_1)|-|E^-(v_1)|=0$ $($mod 3$)$.
By Lemma~\ref{not Z}, $G^{14}$ is not $Z_3$-connected.

If $G\cong G^{15}$, then let $b(v_1)=b(v_2)=b(v_3)=b(v_7)=0$ and $b(v_5)=b(v_6)=1$, $b(v_4)=b(v_8)=2$.
We first orient edges adjacent to vertices $v_1$ and $v_2$. Then in either case, we can orient edges $\{v_4v_3, v_4v_8\}$ both with $v_4$ as a tail by $b(v_4)=2$.
Then we orient edges $v_3v_6, v_3v_5$ both as $v_3$ as a head by $b(v_3)=0$.
Since $b(v_5)=b(v_6)=1$, edges $\{v_5v_6, v_5v_8, v_5v_7\}$ with $v_5$ as a head (or a tail) and $\{v_5v_6, v_7v_6, v_8v_6\}$ all with $v_6$ as a tail (or a head).
In either case, we cannot orient edges $v_7v_8$, such that $|E^+(v_7)|-|E^-(v_7)|=0$ $($mod 3$)$.
By Lemma~\ref{not Z}, $G^{15}$ is not $Z_3$-connected.

If $G\cong G^{17}$, then let $b(v)=0$ for each 3-vertex in $G$ and $b(v)=1$ for each 4-vertex in $G$.
By the similar discussion, we can not find an orientation such that $|E^+(v)|-|E^-(v)|=0$ $($mod 3$)$ for each $v\in V(G)$.
Thus $G^{17}$ is not $Z_3$-connected by Lmmma~\ref{not Z} .
$G^6$ is a subgraph of $G^{17}$. If $G^6$ is $Z_3$-connected, then $G^{17}$ is $Z_3$-connected by Lemma~\ref{o} (6) (7), a contradiction. Thus $G^6$ is not
$Z_3$-connected.
\end{proof}

\section{The case when $\delta(G)\ge 4$}

\begin{lemma}\label{4}
Suppose $G$ is a 3-edge-connected graph with $\delta(G)\ge 4$. If $\alpha (G)\le 2$, then $G$ is $Z_3$-connected.
\end{lemma}
\begin{proof}
Clearly, we can assume $G$ is simple; otherwise, we can contracted $G$ into $G'$ by contracting 2-cycles. By Lemma~\ref{o} (3) (6), $G'$
is $Z_3$-connected if and only if $G$ is $Z_3$-connected.
Since $\delta(G)\ge 4$, $n\ge 5$.
When $n=5$, $G\cong K_5$, by Lemma~\ref{o} (2), $G$ is $Z_3$-connected.
Then we assume $n\ge 6$. By Lemma~\ref{o} (2), We only need to discuss the case $\alpha (G)=2$.

If $d(u)+d(v)\ge n$ for each $uv \notin E(G)$, then
$G$ satisfies the Ore-condition. By Theorem~\ref{ore} and since $\delta(G)\ge 4$, $G$ is $Z_3$-connected.

Thus there exists two non-adjacent vertices  $u, v$ such that $d(u)+d(v)\le n-1$.

Set $x, y$ be such vertices of $G$, that is $d(x)+d(y)\le n-1$ and $xy\notin E(G)$.
Since $\alpha (G)= 2$, $e(v, \{x, y\})\ge 1$ for each $v\in V(G)-\{x, y\}$. Then $|N(x)\cap N(y)|\le 1$ by $d(x)+d(y)\le n-1$.

Case 1. $|N(x)\cap N(y)|=0$.

In this case, $G[N[x]]$ and $G[N[y]]$ is a complete graph $K_{m_1}$,  $K_{m_2}$ ($m_1, m_2\ge 5$) since $\alpha(G)=2$ and $\delta(G)\ge 4$.
By Lemma~\ref{o} (2), $G[N[x]]$ and $G[N[y]]$ is $Z_3$-connected. Since $G$ is
3-edge connected, $G$ is $Z_3$-connected by Lemma~\ref{o} (2) (3) (6).

Case 2. $|N(x)\cap N(y)|=1$.

Suppose $u\in N(x)\cap N(y)$. Similarly, we know that $G[N[x]-\{u\}]$ and $G[N[y]-\{u\}]$ is a complete graph.
Suppose $G[N[x]-\{u\}]\cong K_{m_1}$, $G[N[y]-\{u\}]\cong K_{m_2}$. Clearly, $m_i\ge 4$ for each $i\in \{1, 2\}$.

If $m_i=4$ for each $i\in \{1, 2\}$, then $G[N[x]-\{u\}]\cong K_4$, $G[N[y]-\{u\}]\cong K_4$. Suppose $N(x)=\{x_1, x_2, x_3, u\}$ and $N(y)=\{y_1, y_2, y_3, u\}$.
If $e(u, N[x])\ge 3$, then $N[x]$ contains a $K_5^-$ as a subgraph,
by Lemma~\ref{o} (2), $G[N[x]]$ is $Z_3$-connected. Since $G$ is 3-edge-connected, $e(N[x], N[y]-\{u\})\ge 3$, by Lemma~\ref{o} (2) (6), $G$ is $Z_3$-connected.
WLOG, we assume $1 \le e(u, N[x])\le 2$ and $1 \le e(u, N[y])\le 2$. Then there are at least two vertices in $\{x_1, x_2, x_3\}$ which is not adjacent to $u$ and at least
two vertices in $\{y_1, y_2, y_3\}$ which is not adjacent to $u$. WLOG, we assume $x_1, x_2, y_1, y_2\notin N(u)$. In this case, $x_iy_j\in E(G)$ for each $i, j\in \{1, 2\}$ since $\alpha (G)=2$.
Then we can get a trivial graph $K_1$ by contracting 2-cycles from $G_{[yy_1, yy_2]}$. By Lemma~\ref{o} (3) (6), $G_{[yy_1, yy_2]}$ is $Z_3$-connected.
By Lemma~\ref{1}, $G$ is $Z_3$-connected.

If $m_1=4$ and $m_2\ge 5$, then $G[N[x]-\{u\}]\cong K_4$ and $G[N[y]-\{u\}]$ is $Z_3$-connected by Lemma~\ref{o} (2).
If $e(u, N[y]-\{u\})\ge 2$, then $G[N[y]]$ is $Z_3$-connected by Lemma~\ref{o} (7).
Since $G$ is 3-edge-connected, $e(N[x]-\{u\}, N[y])\ge 3$, by Lemma~\ref{o} (2) (6), $G$ is $Z_3$-connected. Suppose $e(u, N[y]-\{u\})=1$. If there exist $v\in N(x)$ such that
$vu\notin E(G)$, then $e(v, N(y))=m_2-1\ge 4$. Thus $G[N[y]\cup \{v\}-\{u\}]$ is $Z_3$-connected by Lemma~\ref{o} (7).
Since $G$ is 3-edge connected, $G$ is $Z_3$-connected by Lemma~\ref{o} (2) (6).
Thus $e(u, N[x]-\{u\})=4$, this means $G[N[x]]$ is  $K_5$, by Lemma~\ref{o} (2),  $G[N[x]]$ is $Z_3$-connected.
Since $G$ is 3-edge connected, $G$ is $Z_3$-connected by  Lemma~\ref{o} (6) (7).

If $m_i\ge 5$ for each $i\in \{1, 2\}$, then $G[N[x]-\{u\}]$ and $G[N[y]-\{u\}]$ is $Z_3$-connected by Lemma~\ref{o} (2).
Since $G$ is 3-edge connected, $G$ is $Z_3$-connected by Lemma~\ref{o} (6) (7).
\end{proof}

\section{Proof of Theorem~\ref{th1}}

In this section, we define $\cal F$ be a family of 3-edge connected simple graphs $G$ which satisfies
$\alpha (G)= 2$ and $\delta(G)=3$.

\begin{lemma}\label{n-2}
Suppose $G\in \cal F$. If there exists two non-adjacent vertices $u, v$ such that $d(u)+d(v)=n-2$,
then either $G$ is one of the graphs $\{G^{15}, G^{16}, G^{17}, G^{18}\}$ shown in Fig. 1 or $G$ can be $Z_3$-contracted in to $\{K_1, K_4\}$.
\end{lemma}
\begin{proof}
Set $x, y$ be such vertices of $G$, that is $d(x)+d(y)= n-2$ and $xy\notin E(G)$.
Since $\alpha (G)= 2$, $e(v, \{x, y\})\ge 1$ for each $v\in V(G)-\{x, y\}$.
Since $d(x)+d(y)= n-2$, $|N(x)\cap N(y)|=0$.

In this case, $G[N[x]]$ and $G[N[y]]$ is a complete graph $K_{m_1}$,  $K_{m_2}$ ($m_1, m_2\ge 4$) by $\alpha(G)=2$ and $\delta(G)=3$.
If $m_i\ge 5$ for each $i\in \{1, 2\}$, then $\delta(G)\ge 4$, contrary to $\delta (G)=3$.
Thus we need to discuss cases of $m_i=4$ for some $i\in \{1, 2\}$. Assume $m_1=4$. Let $N(x)=\{x_1, x_2, x_3\}$.

Case 1. $m_2=4$.

Let $N(y)=\{y_1, y_2, y_3\}$. Since $G$ is 3-edge connected, $e(N(x), N(y))\ge 3$. If there exists one vertex, say $x_1$, such that $e(x_1, N(y))\ge 3$, then $G[N[y]\cup {x_1}]$
is $Z_3$-connected by Lemma~\ref{o} (2). In this case, if $e(v, N[y]\cup \{x_1\})=1$ for each $v\in N[x]-\{x_1\}$, then $G$ can be contracted into $K_4$; otherwise, $G$ is $Z_3$-connected.
Thus we assume $e(u, N[y])\le 2$ and $e(v, N[x])\le 2$ for each $u\in N(x)$ and $v\in N(y)$.

When $e(N(x), N(y))= 3$. Then $G$ is one of graphs $\{G^{15}, G^{16}, G^{17}\}$ in Fig. 1.

When $e(N(x), N(y))= 4$. Then either $G\cong G^{18}$ in Fig. 1 or $G$ contains a 4-vertex. In the latter case, we assume $d(x_1)=4$ and
$N(x_1)=\{x, x_2, x_3, y_1\}$. Then $e(\{x, x_2, x_3\}, N(y))=3$. Considering graph $G_{[x_1x, x_1x_2]}$.
Clearly, $\{x, x_2, x_3\}$ can be contracted into one vertex $v^*$ by contracting two 2-cycles and we called this new graph $G^*$. Since
$e(\{x, x_2, x_3\}, N(y))=3$, $d_{G^*}(v^*)=3+1=4$. That is $G^*$ contains a $K_5^-$ or $C_2$ as a subgraph. In either case, by Lemma~\ref{o} (2) (3) (6) (7), $G^*$ is $Z_3$-connected.
By Lemma~\ref{o} (3) (6), $G_{[x_1x, x_1x_2]}$ is $Z_3$-connected. Thus $G$ is $Z_3$-connected by Lemma~\ref{1}.

When $e(N(x), N(y))\ge 5$. We only need to prove the case  $e(N(x), N(y))= 5$.
In this case, we may assume $e(x_i, N(y))=2$ for each $i=1, 2$ and $e(x_3, N(y))=1$. WLOG, we assume $x_1y_1, x_1y_2\in E(G)$.
We can get a trivial graph $K_1$ by contracting 2-cycles from graph $G_{[x_1y_1, x_1y_2]}$.
By Lemma~\ref{o} (1) (3) (6), $G_{[x_1y_1, x_1y_2]}$ is $Z_3$-connected.
Thus $G$ is $Z_3$-connected by Lemma~\ref{1}.

Case 2. $m_2\ge 5$.

By Lemma~\ref{o} (2), $G[N[y]]$ is $Z_3$-connected. Since $G$ is 3-edge connected, $e(N(x), N(y))\ge 3$.
We can deduce that $G$ can be contracted to $K_4$ or $G$ is $Z_3$-connected by Lemma~\ref{o} (2) (6).
\end{proof}

\begin{lemma}\label{n-1}
Suppose $G\in \cal F$ and $d(u)+d(v)\ge n-1$ for each $uv\notin E(G)$. If there exists two non-adjacent vertices $u, v$ such that $d(u)+d(v)=n-1$,
then either $G$ is one of the graphs $\{G^6, G^7, \ldots, G^{14}\}$ shown in Fig. 1 or $G$ can be $Z_3$-contracted in to $\{K_1, K_4\}$.
\end{lemma}
\begin{proof}
Set $x, y$ be such vertices of $G$, that is $d(x)+d(y)= n-1$ and $xy\notin E(G)$.
Since $\alpha (G)= 2$, $e(v, \{x, y\})\ge 1$ for each $v\in V(G)-\{x, y\}$.
Since $d(x)+d(y)= n-1$, $|N(x)\cap N(y)|=1$.

Suppose $u\in N(x)\cap N(y)$. Similarly, we know that $G[N[x]-\{u\}]$ and $G[N[y]-\{u\}]$ is a complete graph.
Suppose $G[N[x]-\{u\}]\cong K_{m_1}$, $G[N[y]-\{u\}]\cong K_{m_2}$. Clearly, $m_i\ge 3$ for each $i\in \{1, 2\}$.

Case 1. Suppose $m_i=3$ for each $i=1, 2$.

Let $N(x)=\{x_1, x_2, u\}$, $N(y)=\{y_1, y_2, u\}$. Since $\delta(G)= 3$, $d(u)\ge 3$.

If $d(u)=3$, then WLOG, we assume $N(u)=\{x, x_1, y\}$. Since $\alpha(G)=2$, $x_2y_i\in E(G)$ for each $i\in \{1, 2\}$. Then either $G$ is one
of graphs $\{G^{12}, G^{13}\}$ in Fig. 1 or $x_1y_i\in E(G)$ for each $i\in \{1, 2\}$. In the latter case, we can get a trivial graph $K_1$ by
contracting 2-cycles till no 2-cycles exist from graph $G_{[x_2y_1, x_2y_2]}$. By Lemma~\ref{o} (1) (3) (6), $G_{[x_2y_1, x_2y_2]}$ is $Z_3$-connected.
By Lemma~\ref{1}, $G$ is $Z_3$-connected.

If $d(u)=4$, then $e(u, N[x])=2$ or $e(u, N[x])=3$ by symmetry. Suppose $e(u, N[x])=2$. WLOG, we assume $N(u)=\{x, x_1, y, y_1\}$.
Since $\alpha(G)=2$, $x_2y_2\in E(G)$. If $e(\{x_1, x_2\}, \{y_1, y_2\})=1$, then $G\cong G^{12}$. If $e(\{x_1, x_2\}, \{y_1, y_2\})=2$, then $G\in \{G^8, G^{13}\}$.
If $e(\{x_1, x_2\}, \{y_1, y_2\})\ge 3$, then we can get a trivial graph $K_1$ by contracting 2-cycles till no such subgraph exists from $G_{[ux, ux_1]}$.
By Lemma~\ref{o} (1) (3) (6),
$G_{[ux, ux_1]}$ is $Z_3$-connected. By Lemma~\ref{1}, $G$ is $Z_3$-connected.
Suppose $e(u, N[x])=3$. Then $N(u)=\{x, x_1, x_2, y\}$.
Since $G$ is 3-edge connected, $e(\{x_1, x_2\}, \{y_1, y_2\})\ge 2$.
If $e(\{x_1, x_2\}, \{y_1, y_2\})=2$, then $G$ is one of graphs $\{G^{6}, G^{7}\}$.
If $e(\{x_1, x_2\}, \{y_1, y_2\})\ge 3$, then we can get a trivial graph $K_1$ by contracting 2-cycles till no such subgraph exists from $G_{[ux, ux_1]}$. By Lemma~\ref{o} (1) (3) (6),
$G_{[ux, ux_1]}$ is $Z_3$-connected. By Lemma~\ref{1}, $G$ is $Z_3$-connected.

If $d(u)=5$, then WLOG, we assume $N(u)=\{x, x_1, x_2, y, y_1\}$.
If $e(x_i, \{y_1, y_2\})=2$, then $G$ contains a 4-wheel with $y_1$ as a hub.
Then we can deduce that $G$ is $Z_3$-connected by Lemma~\ref{o} (5) (6) (7).
Thus we assume $e(x_i, \{y_1, y_2\})\le 1$ for each $i\in \{1, 2\}$.
If $e(\{x_1, x_2\}, \{y_1, y_2\})=1$, then $G\cong G^7$. If $e(\{x_1, x_2\}, \{y_1, y_2\})=2$, then either $G\cong G^9$
or we can get a trivial graph $K_1$ by contracting 2-cycles till no such subgraph exists from $G_{[x_ix, x_ix_j]}$, where $x_iy_1\notin E(G)$ and $i\neq j$.
By Lemma~\ref{o} (1) (3) (6),
$G_{[x_ix, x_ix_j]}$ is $Z_3$-connected. By Lemma~\ref{1}, $G$ is $Z_3$-connected..

If $d(u)=6$, then $N(u)=\{x, x_1, x_2, y, y_1, y_2\}$.
If $e(\{x_1, x_2\}, \{y_1, y_2\})=0$, then $G\cong G^{11}$.
If $e(\{x_1, x_2\}, \{y_1, y_2\})=1$, then $G\cong G^{10}$.
Thus we assume $e(\{x_1, x_2\}, \{y_1, y_2\})\ge 2$. If there exists $i$, say $x_1$, such that $e(\{x_1\}, \{y_1, y_2\})= 2$, then $G[N[y]\cup x_1]$ is a
$K_5^-$. By Lemma~\ref{o} (2) (6), we can deduce that $G$ is $Z_3$-connected. Thus WLOG, we assume $x_1y_1, x_2y_2\in E(G)$. Considering graph $G_{[x_1y_1, x_1x_2]}$.
$G_{[x_1y_1, x_1x_2]}$ contains a 4-wheel with $y_1$ as a hub. We can get a new graph with 3 vertices and 4 edges by contracting this 4-wheel from $G_{[x_1y_1, x_1x_2]}$,
which is $Z_3$-connected by Lemma~\ref{o} (5) (6).
By Lemma~\ref{1}, $G$ is $Z_3$-connected.

Case 2. Suppose $m_i=4$ for each $i=1, 2$.

Suppose $N(x)=\{x_1, x_2, x_3, u\}$, $N(y)=\{y_1, y_2, y_3, u\}$. If $e(u, N[x])\ge 3$, then $G[N[x]]$ contains a $K_5^-$ as a subgraph.
By Lemma~\ref{o} (2), $G[N[x]]$ is $Z_3$-connected.
Since $G$ is 3-edge connected, $e(N[x], N[y]-\{u\})\ge 3$. Then $G/G[N[x]]$ contains a $K_5^-$ as a subgraph. By Lemma~\ref{o} (2), $G/G[N[x]]$ is $Z_3$-connected.
By Lemma~\ref{o} (6), $G$ is $Z_3$-connected.
Thus we may assume $1 \le e(u, N[x])\le 2$ and $1 \le e(u, N[y])\le 2$.
Since $G$ is 3-edge connected, $d(u)\ge 3$.

If $e(u, N[x])= 2$ and  $e(u, N[y])= 2$, then we assume, WLOG, we assume $N(u)=\{x, x_1, y, y_1\}$. Since $\alpha(G)=2$, $x_2y_i, x_3y_i\in E(G)$ for each $i=2, 3$.
In this case, $\delta(G)\ge 4$, contrary to $\delta(G)=3$.

If $e(u, N[x])= 2$ and $e(u, N[y])= 1$, then we assume, WLOG, we assume $N(u)=\{x, x_1, y\}$. Since $\alpha(G)=2$, $x_2y_i, x_3y_i\in E(G)$ for each $i=1, 2, 3$.
In this case, $G[N[y]\cup \{x_2, x_3\}-\{u\}]$ is $Z_3$-connected by Lemma~\ref{o} (2) (7). By Lemma~\ref{o} (7), $G$ is $Z_3$-connected.

Case 3. Suppose $m_i\ge 5$ for each $i=1, 2$. Clearly, $G[N[x]-\{u\}]$ and $G[N[y]-\{u\}]$ is $Z_3$-connected by Lemma~\ref{o} (2).
Since $\delta(G)= 3$, WLOG, we assume $e(u, N[x]-\{u\})= 2$.
Thus $G[N[x]]$ is $Z_3$-connected by Lemma~\ref{o} (7).
Since $G$ is 3-edge connected, $e(N[x], N[y]-\{u\})\ge 3$, $G$ is $Z_3$-connected by Lemma~\ref{o} (2) (6).

Case 4. Suppose $m_1=3$ and $m_2=4$.

Suppose $N(x)=\{x_1, x_2, u\}$, $N(y)=\{y_1, y_2, y_3, u\}$.
If $e(u, N[y]-\{u\})\ge 3$, then $G[N[y]]$ contains a $K_5^-$ as a subgraph. By Lemma~\ref{o} (2), $G[N[y]]$ is $Z_3$-connected.
Since $G$ is 3-edge connected, $e(N[x]-\{u\}, N[y])\ge 3$. Then either $G$ can be contracted to $K_4$ or $G$ is $Z_3$-connected.
Thus we assume $1 \le e(u, N[y]-\{u\})\le 2$ and $d(u)\ge 3$.

If $d(u)=3$, then $e(u, N[x]-\{u\})=1$ and $e(u, N[y]-\{u\})=2$ or
$e(u, N[x]-\{u\})=2$ and  $e(u, N[y]-\{u\})=1$.
In the former case, we assume, WLOG, $N(u)=\{x, y, y_1\}$. Since $\alpha(G)=2$, $x_1y_i, x_2y_i\in E(G)$ for each $i=2, 3$.
In this case, we can get a trivial graph $K_1$ by contracting 2-cycles from $G_{[x_1y_2, x_1y_3]}$.
By Lemma~\ref{o} (1) (3) (6), $G_{[x_1y_2, x_1y_3]}$ is $Z_3$-connected. By Lemma~\ref{1}, $G$ is $Z_3$-connected.
In the latter case, we assume, WLOG, $N(u)=\{x_1, x, y\}$. Since $\alpha(G)=2$, $x_2y_i\in E(G)$ for each $i=1, 2, 3$.
In this case, $G[N[y]\cup \{x_2\}-\{u\} ]$ is $Z_3$-connected by Lemma~\ref{o} (2). Then either $G$ is $Z_3$-connected or $G$ can contracted into $K_4$.

If $d(u)=4$, then $e(u, N[x]-\{u\})= 2$ and  $e(u, N[y]-\{u\})= 2$ or $e(u, N[x]-\{u\})= 3$ and  $e(u, N[y]-\{u\})= 1$.
Suppose $e(u, N[x]-\{u\})=2$ and  $e(u, N[y]-\{u\})=2$. We assume, WLOG, $N(u)=\{x, x_1, y, y_1\}$. Since $\alpha(G)=2$, $x_2y_i\in E(G)$ for each $i=2, 3$.
If no other edge, then $G\cong G^{14}$; otherwise, either we can get a trivial graph $K_1$ by contracting 2-cycles from $G_{[uy, uy_1]}$ or
$G{[N[y]]\cup \{x_2\}-\{u\}}$ contains a $K_5^-$ as a subgraph.
In the former case, by Lemma~\ref{o} (1) (3) (6), $G_{[uy, uy_1]}$ is $Z_3$-connected. By Lemma~\ref{1}, $G$ is $Z_3$-connected; in the latter case,
$G{[N[y]]\cup \{x_2\}-\{u\}}$ is $Z_3$-connected by Lemma~\ref{o} (2). By Lemma~\ref{o} (7), $G$ is also $Z_3$-connected.
Then suppose $e(u, N[x]-\{u\})= 3$ and $e(u, N[y]-\{u\})= 1$. Then $N(u)=\{x_1, x_2, x, y\}$.
Since $G$ is 3-edge connected, $e(\{x_1, x_2\}, \{y_1, y_2, y_3\})\ge 2$.
Since $d(u)+d(v)\ge n-1=7$ for each $uv\notin E(G)$, $d(y_i)\ge 4$ for each $i=1, 2, 3$. In this case, $e(\{x_1, x_2\}, \{y_1, y_2, y_3\})\ge 3$.
Then either we can get a trivial graph $K_1$ by contracting 2-cycles from $G_{[yy_1, yy_2]}$
or $G[N[y]\cup \{x_i\}-\{u\}]$ contains $K_5^-$ as spanning subgraph for some $i\in \{1, 2\}$.
In the former case, $G_{[yy_1, yy_2]}$ is $Z_3$-connected by Lemma~\ref{o} (1) (3) (6). By Lemma~\ref{1}, $G$ is $Z_3$-connected;
in the latter case, we can deduce $G$ is $Z_3$-connected by Lemma~\ref{o} (2) (7).

If $d(u)=5$, then $e(u, N[x]-\{u\})= 3$ and  $e(u, N[y]-\{u\})= 2$. WLOG, we assume $N(u)=\{x_1, x_2, x, y, y_1\}$.
Since $d(u)+d(v)\ge n-1=7$ for each $uv\notin E(G)$, $d(y_i)\ge 4$ for each $i=1, 2, 3$.
Then $e(\{x_1, x_2\}, \{y_1, y_2, y_3\})\ge 2$.
In this case, we can get a trivial graph $K_1$ by contracting 2-cycles and $K_5^-$ from graph $G_{[yy_1, yy_2]}$.
By Lemma~\ref{o} (1) (6), $G_{[yy_1, yy_2]}$ is $Z_3$-connected. By Lemma~\ref{1}, $G$ is $Z_3$-connected.

Case 5. Suppose $m_1=3$ and $m_2\ge 5$.

Clearly, $G[N[y]-\{u\}]$ is $Z_3$-connected by Lemma~\ref{o} (2). Since $\delta(G)\ge 3$, $e(u, N[x]-\{u\})\ge 2$ or $e(u, N[y]-\{u\})\ge 2$.
If $e(u, N[y]-\{u\})\ge 2$, then $G[N[y]]$ is $Z_3$-connected by Lemma~\ref{o} (7).
Since $G$ is 3-edge connected, $e(N[x]-\{u\}, N[y])\ge 3$.
Thus either $G$ is $Z_3$-connected or $G$ can be contracted into $K_4$.
Thus we assume $e(u, N[y]-\{u\})=1$ and $e(u, N[x]-\{u\})\ge 2$. Set $N(x)=\{x_1, x_2, u\}$. WLOG, we assume $x_1u\in E(G)$.
If $x_2u\in E(G)$, then $G[N[x]]$ is $K_4$. Since $G$ is 3-edge connected, $e(N[x], N[y]-\{u\})\ge 3$. Thus $G/G[N[y]-\{u\}]$ contains $K_5^-$ or $C_2$.
In either case, $G$ is $Z_3$-connected by Lemma~\ref{o} (2) (3) (6). If $x_2u\notin E(G)$, then
$x_2v\in E(G)$ for each $v\in N(y)-\{u\}$ by $\alpha(G)=2$. Thus $G[N[y]\cup \{x_2\}-\{u\}]$ is $Z_3$-connected by Lemma~\ref{o} (7).
Thus either $G$ is $Z_3$-connected or $G$ can be contracted into $K_4$.

Case 6. Suppose $m_1=4$ and $m_2\ge 5$.

Clearly, $G[N[y]-\{u\}]$ is $Z_3$-connected by Lemma~\ref{o} (2). Since $\delta(G)\ge 3$, $e(u, N[x]-\{u\})\ge 2$ or $e(u, N[y]-\{u\})\ge 2$.
If $e(u, N[y]-\{u\})\ge 2$, then $G[N[y]]$ is $Z_3$-connected by Lemma~\ref{o} (7).
Since $G$ is 3-edge connected, $e(N[x]-\{u\}, N[y])\ge 3$.
Thus $G/G[N[y]]$ contains $K_5^-$ or $C_2$.
In either case, $G$ is $Z_3$-connected by Lemma~\ref{o} (1) (2) (3) (6).
Thus $e(u, N[y]-\{u\})=1$ and $e(u, N[x]-\{u\})\ge 2$. Set $N(x)=\{x_1, x_2, x_3, u\}$. WLOG, we assume $x_1u\in E(G)$.
If $x_iu\in E(G)$ for some $i\in \{2, 3\}$, then $G[N[x]]$ contains $K_5^-$ as a subgraph, is $Z_3$-connected. Since $G$ is 3-edge connected, $e(N[x], N[y]-\{u\})\ge 3$.
Thus $G/G[N[x]]$ contains $K_5^-$ or $C_2$. By Lemma~\ref{o} (2) (3) (6), $G$ is $Z_3$-connected.
If  $x_iu\notin E(G)$ for each $i\in \{2, 3\}$, then
$x_iv\in E(G)$ for each $v\in N(y)-\{u\}$ and $i\in \{2, 3\}$ by $\alpha(G)=2$. Thus $G[N[y]\cup \{x_2, x_3\}-\{u\}]$ is $Z_3$-connected by Lemma~\ref{o} (7).
By Lemma~\ref{o} (7), $G$ is $Z_3$-connected.
\end{proof}

\vskip 0.5cm

\n \textbf{Proof of Theorem 1.3}
Since $G$ a is 3-edge connected simple graph, $n\ge 4$.
When $n=4$, $G\cong K_4$, that is $G^1$. When $n=5$, $G$ contains $W_4$ as a subgraph, by Lemma~\ref{o} (1) (5) (6), $G$ is $Z_3$-connected.

By Lemma~\ref{n-2} and ~\ref{n-1}, we only need prove the case of $G$ which satisfies the Ore-condition.
By Theorem~\ref{ore} and since $G$ is 3-edge-connected, either $G$ is $Z_3$-connected or $G$ is one of graphs $\{G^2, G^3, G^4, G^5\}$ shown in Fig. 1.
Thus we prove that $G$ is not one of the 18 special graphs shown in Fig. 1
if and only if  $G$ can be $Z_3$-contracted to one of the graphs $\{K_1,
K_4\}$.
$\blacksquare$

\end{document}